\documentclass[a4paper,oneside,12pt,reqno]{amsart}

\title{Rigidity of free boundary MOTS}

\author{Abraão Mendes}

\usepackage{lmodern}
\usepackage{amsmath}
\usepackage{amsthm}
\usepackage{amssymb}

\newtheorem{thm}{Theorem}[section]
\newtheorem{lemma}[thm]{Lemma}
\newtheorem{proposition}[thm]{Proposition}

\theoremstyle{remark}
\newtheorem{remark}[thm]{Remark}

\DeclareMathOperator{\divergente}{div}

\newcommand{\p}{\partial}
\newcommand{\R}{\mathbb{R}}

\renewcommand{\L}{\mathcal L}
\renewcommand{\div}{\divergente}

\usepackage[margin=2cm]{geometry}

\DeclareMathOperator{\II}{II}
\DeclareMathOperator{\tr}{tr}
\DeclareMathOperator{\real}{Re}

\usepackage{setspace}

\renewcommand{\sl}{\textsl}
\renewcommand{\Re}{\real}

\address{Instituto de Matemática, Universidade Federal de Alagoas, Maceió, AL, Brazil}
\email{abraao.mendes@im.ufal.br}

\begin{document}

\onehalfspacing

\begin{abstract}
The aim of this work is to present an initial data version of Hawking's theorem on the topology of back hole spacetimes in the context of manifolds with boundary. More precisely, we generalize the results of G. J. Galloway and R. Schoen \cite{GallowaySchoen} and G. J. Galloway \cite{Galloway2008,Galloway2018} by proving that a compact free boundary stable marginally outer trapped surface (MOTS) $\Sigma$ in an initial data set with boundary satisfying natural dominant energy conditions (DEC) is of positive Yamabe type, i.e. $\Sigma$ admits a metric of positive scalar curvature with minimal boundary, provided $\Sigma$ is outermost. To do so, we prove that if $\Sigma$ is a compact free boundary stable MOTS which does not admit a metric of positive scalar curvature with minimal boundary in an initial data set satisfying the interior and the boundary DEC, then an outer neighborhood of $\Sigma$ can be foliated by free boundary MOTS $\Sigma_t$, assuming that $\Sigma$ is weakly outermost. Moreover, each $\Sigma_t$ has vanishing outward null second fundamental form, is Ricci flat with totally geodesic boundary, and the dominant energy conditions saturate on $\Sigma_t$.
\end{abstract}

\maketitle

\section{Introduction}

In a seminal paper, S. W. Hawking \cite{Hawking} showed, among other things, that closed cross-sections of the event horizon in $(3+1)$-dimensional asymptotically flat stationary black hole spacetimes obeying the \sl{spacetime} dominant energy condition (DEC) are topologically $2$-spheres. Hawking's theorem extends to outer apparent horizons, i.e. \sl{weakly outermost} marginally outer trapped surfaces (MOTS) in more general (not necessarily stationary) black hole spacetimes.

In \cite{GallowaySchoen}, G. J. Galloway and R. Schoen extended Hawking's theorem to higher dimensional black hole spacetimes by proving that cross-sections of the event horizon (in the stationary case) and outer apparent horizons (in the general case) are of positive Yamabe type, i.e. they admit metrics of positive scalar curvature. Their result is consistent with the example presented by R.~Emparan and H. S. Reall \cite{EmparanReall} of a $(4+1)$-dimensional stationary vacuum black hole spacetime with horizon topology $S^2\times S^1$. 

We paraphrase Galloway-Schoen's result as follows. Definitions are given in Section \ref{section.2}.

\begin{thm}[\cite{GallowaySchoen}]\label{thm.GallowaySchoen}
Let $(M,g,K)$ be an $(n+1)$-dimensional, $n\ge2$, initial data set satisfying the DEC. If $\Sigma$ is a closed stable MOTS in $(M,g,K)$, then either:
\begin{enumerate}
\item $\Sigma$ admits a metric of positive scalar curvature or
\item $\Sigma$ is Ricci flat, the DEC saturates on $\Sigma$, and $\Sigma$ has vanishing outward null second fundamental form.\label{item.2}
\end{enumerate}
\end{thm}

Assuming the initial data set $(M,g,K)$ can be embedded into a spacetime obeying the spacetime DEC and $\Sigma$ is a closed \sl{outermost} MOTS in $(M,g,K)$, Galloway \cite{Galloway2008} was able to rule out the `exceptional circumstances' (item \eqref{item.2}) of Theorem \ref{thm.GallowaySchoen}; that is, under these hypotheses, $\Sigma$ in fact admits a metric of positive scalar curvature.

A few years ago, Galloway \cite{Galloway2018} was able to drop the embeddedness assumption on the initial data set and prove a purely initial data version of the results presented in \cite{Galloway2008} and \cite{GallowaySchoen}. Among other things, he proved the following rigidity result.

\begin{thm}[\cite{Galloway2018}]\label{thm1.Galloway2018}
Let $(M,g,K)$ be an $(n+1)$-dimensional, $n\ge2$, initial data set satisfying the DEC. Suppose $\Sigma$ is a closed weakly outermost MOTS in $(M,g,K)$ that does not admit a metric of positive scalar curvature. Then there exists an outer neighborhood $U\cong[0,\epsilon)\times\Sigma$ of $\Sigma$ in $M$ such that, for each $t\in[0,\epsilon)$, the following hold:
\begin{enumerate}
\item $\Sigma_t\cong\{t\}\times\Sigma$ is a MOTS. In fact, $\Sigma_t$ has vanishing outward null second fundamental form.
\item $\Sigma_t$ is Ricci flat with respect to the induced metric.
\item The DEC saturates on $\Sigma_t$.
\end{enumerate}
\end{thm}

As an immediate consequence of Theorem \ref{thm1.Galloway2018}, he obtained the following:

\begin{thm}[\cite{Galloway2018}]\label{thm2.Galloway2018}
Let $(M,g,K)$ be an $(n+1)$-dimensional, $n\ge2$, initial data set satisfying the DEC. If $\Sigma$ is a closed outermost MOTS in $(M,g,K)$, then $\Sigma$ is of positive Yamabe type, i.e. $\Sigma$ admits a metric of positive scalar curvature.
\end{thm}

Our main goal in the present work is to extend Theorems \ref{thm1.Galloway2018} and \ref{thm2.Galloway2018} to the context of compact \sl{free boundary} MOTS in initial data sets with boundary satisfying the \sl{interior} and the \sl{boundary} dominant energy conditions.

Our first result is the following:

\begin{thm}\label{thm.4}
Let $(M,g,K)$ be an $(n+1)$-dimensional, $n\ge2$, initial data set with boundary and $\Sigma$ be a compact free boundary stable MOTS in $(M,g,K)$. Suppose $(M,g,K)$ satisfies the interior and the boundary DEC. If $\Sigma$ does not admit a metric of positive scalar curvature with minimal boundary and is weakly outermost in $(M,g,K)$, then there exists an outer neighborhood $U\cong[0,\epsilon)\times\Sigma$ of $\Sigma$ in $M$ such that, for each $t\in[0,\epsilon)$, the following hold: 
\begin{enumerate}
\item $\Sigma_t\cong\{t\}\times\Sigma$ is a free boundary MOTS. In fact, $\Sigma_t$ has vanishing outward null second fundamental form.
\item $\Sigma_t$ is Ricci flat and has totally geodesic boundary with respect to the induced metric.
\item The interior DEC saturates on $U$ and $J|_{\Sigma_t}=0$.
\item The boundary DEC saturates along $\p\Sigma_t$ and $(\iota_\varrho\pi)^\top|_{\p\Sigma_t}=0$.
\end{enumerate}
\end{thm}

The 1-forms $J$ and $(\iota_{\varrho}\pi)^\top$ are defined in Section \ref{section.2}.

As an immediate consequence of Theorem \ref{thm.4}, we obtain the following topological obstruction result for the existence of free boundary outermost MOTS in initial data sets under natural energy hypotheses, which can be seen as an initial data version of Hawking's theorem in the context of manifolds with boundary.

\begin{thm}
Let $(M,g,K)$ be an $(n+1)$-dimensional, $n\ge2$, initial data set with boundary and $\Sigma$ be a compact free boundary stable MOTS in $(M,g,K)$. Suppose $(M,g,K)$ satisfies the interior and the boundary DEC. If $\Sigma$ is outermost in $(M,g,K)$, then $\Sigma$ admits a metric of positive scalar curvature with minimal boundary.
\end{thm}

It is important to mention that the compact-with-boundary orientable manifolds which admit metrics of positive scalar curvature with minimal boundary are completely characterized in $n=2,3$ dimensions. In $n=2$, this follows directly from the Gauss-Bonnet theorem. In particular, they are topologically a disk. In $n=3$, the characterization of such manifolds was done by A. Carlotto and C. Li \cite{CarlottoLi2019}.

This paper is organized as follows. In Section \ref{section.2}, we review some background material on MOTS. In Section \ref{section.3}, we state and prove several auxiliary results. In Section \ref{section.4}, we present the proof of Theorem \ref{thm.4}. Finally, in Section \ref{section.5}, we establish a splitting result under natural convexity and volume-minimizing conditions. 

\medskip

{\bf Acknowledgments.} The author would like to thank Márcio H. Batista and Feliciano Vitório for their kind interest in this work. He would also like to thank Marcos P. Cavalcante, Tiarlos Cruz, and Gregory J. Galloway for their valuable comments regarding this paper. The author was partially supported by the National Council for
Scientific and Technological Development – CNPq (Grant 305710/2020-6).

\section{Marginally Outer Trapped Surfaces}\label{section.2}

An \sl{initial data set} $(M,g,K)$ consists of a Riemannian manifold $(M,g)$ and a symmetric $(0,2)$-tensor $K$ on $M$. We assume that $M$ is oriented throughout
this paper. 

The \sl{local energy density} $\mu$ and the \sl{local current density} $J$ of $(M,g,K)$ are given by
\begin{align*}
\mu=\frac{1}{2}\left(R-|K|^2+(\tr K)^2\right)\,\,\,\,\mbox{and}\,\,\,\,J=\div(K-(\tr K)g),
\end{align*}
where $R$ is the scalar curvature of $(M,g)$. The initial data set $(M,g,K)$ is said to satisfy the \sl{dominant energy condition} (DEC) if 
\begin{align}\label{eq.interior.DEC}
\mu\ge|J|\,\,\,\,\mbox{on}\,\,\,\,M.
\end{align}

The \sl{momentum tensor} $\pi$ of $(M,g,K)$ is given by
\begin{align*}
\pi=K-(\tr K)g.
\end{align*}
Then $\mu$ and $J$ can be written in terms of $\pi$ by
\begin{align*}
\mu=\frac{1}{2}\left(R-|\pi|^2+\frac{1}{n}(\tr\pi)^2\right)\,\,\,\,\mbox{and}\,\,\,\,J=\div\pi,
\end{align*}
where $n+1$ is the dimension of $M$, $n\ge2$.

When $M$ is a manifold with boundary, we denote the second fundamental form of $\p M$ in $(M,g)$ by $\II^{\p M}$. More precisely,
\begin{align*}
\II^{\p M}(Y,Z)=\langle\nabla_Y\varrho,Z\rangle,\,\,\,\,Y,Z\in\mathfrak{X}(\p M),
\end{align*}
where $\varrho$ is the outward unit normal of $\p M$ in $(M,g)$. The mean curvature $H^{\p M}$ of $\p M$ in $(M,g)$ is given by
\begin{align*}
H^{\p M}=\tr\II^{\p M}=\div_{\p M}\varrho.
\end{align*}
We define $(\iota_{\varrho}\pi)^\top$ as the restriction of $\iota_{\varrho}\pi=\pi(\varrho,\cdot)$ to tangent vector fields of $\p M$ and say that $(M,g,K)$ satisfies the \sl{boundary DEC} if 
\begin{align*}
H^{\p M}\ge|(\iota_{\varrho}\pi)^\top|\,\,\,\,\mbox{along}\,\,\,\,\p M.
\end{align*} 

The above boundary DEC was introduced by S. Almaraz, L. L. Lima, and L. Mari \cite{AlmarazLimaMari} in the context of proving spacetime positive mass inequalities for asymptotically flat and asymptotically hyperbolic initial data sets with noncompact boundary.

In order to avoid ambiguity, we call \eqref{eq.interior.DEC} the \sl{interior DEC}.

Let $\Sigma$ be a connected two-sided hypersurface in $(M,g)$ with unit normal $N$ and 
\begin{align*}
H=\div_\Sigma N
\end{align*}
be its associated mean curvature. The \sl{null mean curvatures} $\theta^+$ and $\theta^-$ of $\Sigma$ in $(M,g,K)$ are defined by
\begin{align*}
\theta^+=\tr_\Sigma K+H\,\,\,\,\mbox{and}\,\,\,\,\theta^-=\tr_\Sigma K-H.
\end{align*}
Then the hypersurface $\Sigma$ is said to be \sl{outer trapped} if $\theta^+<0$, \sl{weakly outer trapped} if $\theta^+\le0$, and \sl{marginally outer trapped} if $\theta^+=0$. In the latter case, we refer to $\Sigma$ as a \sl{marginally outer trapped surface} (MOTS).

The \sl{null second fundamental forms} $\chi^+$ and $\chi^-$ of $\Sigma$ in $(M,g,K)$ are defined by
\begin{align*}
\chi^+=K|_\Sigma+A\,\,\,\,\mbox{and}\,\,\,\,\chi^-=K|_\Sigma-A,
\end{align*}
where $A$ is the second fundamental form of $\Sigma$ in $(M,g)$. Our sign convention is such that $H=\tr A$ and, in particular, $\theta^\pm=\tr\chi^\pm$.

\begin{remark}
Initial data sets arise naturally in general relativity. In fact, let $(\bar M,\bar g)$ be a spacetime, i.e. a time-oriented Lorentzian manifold. Consider a spacelike hypersurface $M$ in $(\bar M,\bar g)$ and let $g$ be the induced Riemannian metric on $M$ and $K$ be the second fundamental form of $M$ with respect to the future-pointing timelike unit normal $u$ of $M$ in $(\bar M,\bar g)$. Then $(M,g,K)$ is an initial data set. As before, let $\Sigma$ be a two-sided hypersurface in $(M,g)$ with unit normal~$N$. In this setting, $\chi^+$ and $\chi^-$ are the second fundamental forms of $\Sigma$ in $(\bar M,\bar g)$ with respect to the \sl{null normal fields}
\begin{align*}
\ell^+=u|_\Sigma+N\,\,\,\,\mbox{and}\,\,\,\,\ell^-=u|_\Sigma-N,
\end{align*}
respectively. Therefore $\theta^+=\tr\chi^+=\div_\Sigma\ell^+$ is the \sl{future outgoing null expansion scalar} and $\theta^-=\tr\chi^-=\div_\Sigma\ell^-$ is the \sl{future ingoing null expansion scalar} of $\Sigma$ in $(\bar M,\bar g)$. 
\end{remark}

Let $(\Sigma_t)_{|t|<\epsilon}$ be a variation of $\Sigma$ in $M$, $\Sigma=\Sigma_0$, with variation vector field
\begin{align*}
\mathcal V=\frac{\p}{\p t}\big|_{t=0}=\phi N\,\,\,\,\mbox{for some}\,\,\,\,\phi\in C^\infty(\Sigma).
\end{align*}
We may view the null mean curvature $\theta^+$ of $\Sigma_t$ in $(M,g,K)$ as a parameter-dependent function on $\Sigma$. Computations as in \cite{AnderssonEichmairMetzger} show that 
\begin{align}\label{eq.first.variation.theta}
\frac{d}{dt}\big|_{t=0}\theta^+(t,\cdot)=L\phi+\left(-\frac{1}{2}(\theta^+)^2+\theta^+\tau\right)\phi,
\end{align}
where
\begin{align*}
L\phi\triangleq-\Delta\phi+2\langle X,\nabla\phi\rangle+(Q-|X|^2+\div X)\phi
\end{align*}
and
\begin{align*}
Q\triangleq\frac{1}{2}R^\Sigma-(\mu+J(N))-\frac{1}{2}|\chi^+|^2.
\end{align*}
Here, $\Delta$ is the nonpositive Laplace-Beltrami operator, $\nabla$ the gradient, $\div$ the divergence, and $R^\Sigma$ the scalar curvature of $\Sigma$ with respect to the induced metric. Moreover, $X$ is the tangent vector field of $\Sigma$ that is dual to the $1$-form $K(N,\cdot)|_{\Sigma}$, and $\tau=\tr K$. 

When $\Sigma$ is a MOTS, $L$ is called the \sl{MOTS stability operator} of $\Sigma$. This is because in the Riemannian case ($K=0$), saying that $\Sigma$ is a MOTS is equivalent to saying that $\Sigma$ is a minimal hypersurface and, in this case, $L$ reduces to the classical stability operator of minimal surfaces theory. 

It is possible to show that, when $\Sigma$ is a closed MOTS, the operator $L$ admits a real eigenvalue $\lambda_1=\lambda_1(L)$, called its \sl{principal eigenvalue}, such that $\lambda_1\le\Re\lambda$ for any other eigenvalue~$\lambda$. Moreover, the associated eigenfunction $\phi_1$, $L\phi_1=\lambda_1\phi_1$, is unique up to a multiplicative constant, and $\phi_1$ can be chosen to be positive. Also, $\lambda_1$ is nonnegative if, and only if, there exists a positive function $\phi\in C^\infty(\Sigma)$ such that $L\phi\ge0$ (see \cite{AnderssonMarsSimon}). In analogy with the minimal surface case, a closed MOTS $\Sigma$ is said to be \sl{stable} provided $\lambda_1(L)\ge0$. 

Now, we are going to present a notion of \sl{stability} for \sl{capillary} MOTS in initial data sets with boundary introduced by A. Alaee, M. Lesourd, and S.-T. Yau \cite{AlaeeLesourdYau}. We restrict our attention to the \sl{free boundary} case.

Assume that $M$ is a manifold with boundary and $\Sigma$ is a compact-with-boundary \sl{properly embedded} hypersurface in $M$, which means that $\Sigma$ is embedded in $M$ and $\p\Sigma=\Sigma\cap\p M$. We say that $\Sigma$ is \sl{free boundary} in $(M,g)$ if $\Sigma$ meets $\p M$ orthogonally, that is, $\varrho=\nu$ along $\p\Sigma$, where $\nu$ is the outward unit normal of $\p\Sigma$ in $\Sigma$ with respect to the induced metric.

\begin{remark}
The mean curvature $H^{\p\Sigma}$ of $\p\Sigma$ in $\Sigma$ with respect to the induced metric is given by
\begin{align*}
H^{\p\Sigma}=\div_{\p\Sigma}\nu.
\end{align*}
In particular, if $\Sigma$ is free boundary (equivalently, $N$ is tangent to $\p M$ along $\p\Sigma$), then
\begin{align*}
H^{\p M}=\div_{\p M}\varrho=\div_{\p\Sigma}\nu+\langle\nabla_N\varrho,N\rangle=H^{\p\Sigma}+\II^{\p M}(N,N).
\end{align*}
\end{remark}

Let $(\Sigma_t)_{|t|<\epsilon}$ be a variation of $\Sigma=\Sigma_0$ in $(M,g,K)$ by compact-with-boundary properly embedded hypersurfaces and consider the functional
\begin{align*}
\mathcal F[\Sigma_t]=\int_{\Sigma_t}\theta^+(t)\langle\mathcal V,N_t\rangle dv_t+\int_{\p\Sigma_t}\langle\mathcal V,\nu_t\rangle ds_t,
\end{align*}
where
\begin{align*}
\mathcal V=\mathcal V_t=\frac{\p}{\p t}.
\end{align*}
If $\Sigma$ is a free boundary MOTS in $(M,g,K)$, then 
\begin{align}\label{eq.F.functional}
\frac{d}{dt}\big|_{t=0}\mathcal F[\Sigma_t]=\int_\Sigma\phi L\phi\hspace{.01cm}dv-\int_{\p\Sigma}\phi\left(\frac{\p\phi}{\p\nu}-\II^{\p M}(N,N)\phi\right)ds,
\end{align}
where $\phi=\langle\mathcal V,N\rangle$ (see \cite[Section 5]{AlaeeLesourdYau}). The corresponding eigenvalue problem to functional \eqref{eq.F.functional} is the following:

\begin{align*}
L\phi&=-\Delta\phi+2\langle X,\nabla\phi\rangle+(Q-|X|^2+\div X)\phi=\lambda\phi\,\,\,\,\mbox{on}\,\,\,\,\Sigma,\\
B\phi&\triangleq\frac{\p\phi}{\p\nu}-\II^{\p M}(N,N)\phi=0\,\,\,\,\mbox{along}\,\,\,\,\p\Sigma.
\end{align*}

Then a compact free boundary MOTS $\Sigma$ in $(M,g,K)$ is said to be \sl{stable} provided there exists a nonnegative function $\phi\in C^\infty(\Sigma)$, $\phi\not\equiv0$, satisfying the Robin boundary condition $B\phi=0$ such that $L\phi\ge0$. Without loss of generality, by the maximum principle for nonpositive functions, we may assume that $\phi>0$.

Before finishing this section, let us present some important definitions. 

Suppose $\Sigma$ is a separating MOTS in $(M,g,K)$. By definition, the unit normal $N$ is said to be \sl{outward-pointing}. We say that $\Sigma$ is \sl{weakly outermost} if there is no outer trapped ($\theta^+<0$) surface outside of, and homologous to, $\Sigma$. Analogously, we say that $\Sigma$ is \sl{outermost} if there is no weakly outer trapped ($\theta^+\le0$) surface outside of, and homologous to, $\Sigma$. 

\section{Auxiliary Results}\label{section.3}

The first auxiliary result we are going to prove is the following.

\begin{lemma}\label{lemma.8}
Let $(M,g,K)$ be an $(n+1)$-dimensional, $n\ge2$, initial data set with boundary and $\Sigma$ be a compact free boundary MOTS in $(M,g,K)$. If $\Sigma$ is stable, then the first eigenvalue $\lambda_1(\L_0)$ of $\L_0\triangleq-\Delta+Q$ on $\Sigma$ with Robin boundary condition $B_0u\triangleq Bu+\langle X,\nu\rangle u=0$ is nonnegative.
\end{lemma}

The proof of Lemma \ref{lemma.8} is essentially contained in the proof of Proposition 5.4 in \cite{AlaeeLesourdYau}. But, for the sake of completeness, we include it here. The inspiring closed case was first proved by Galloway and Schoen \cite{GallowaySchoen} (see also \cite[Section 4]{AnderssonMarsSimon} and \cite[Lemma 2.2]{Galloway2008}).

\begin{proof}[Proof of Lemma \ref{lemma.8}]
Let $\phi\in C^\infty(\Sigma)$ be a positive function satisfying the boundary condition $B\phi=0$ such that $L\phi\ge0$. Observe that  
\begin{align*}
\frac{L\phi}{\phi}&=-\frac{\Delta\phi}{\phi}+2\langle X,\nabla\ln\phi\rangle+Q-|X|^2+\div X\\
&=-\frac{\Delta\phi}{\phi}+|\nabla\ln\phi|^2-|X-\nabla\ln\phi|^2+Q+\div X\\
&=\div Y-|Y|^2+Q,
\end{align*}
where $Y=X-\nabla\ln\phi$. Therefore
\begin{align*}
u^2\frac{L\phi}{\phi}&=\div(u^2Y)-2u\langle Y,\nabla u\rangle-|uY|^2+Qu^2\\
&=\div(u^2Y)-|uY+\nabla u|^2+|\nabla u|^2+Qu^2\\
&\le\div(u^2Y)+|\nabla u|^2+Qu^2,
\end{align*}
for any function $u\in C^\infty(\Sigma)$.

Then, taking $u$ as a first eigenfunction of $\L_0$, $\L_0u=\lambda_1(\L_0)u$, with Robin boundary condition $B_0u=0$ and using that $L\phi\ge0$, we obtain 
\begin{align*}
0&\le\int_\Sigma\left(|\nabla u|^2+Qu^2+\div(u^2Y)\right)dv\\
&=\int_\Sigma\left(|\nabla u|^2+Qu^2\right)dv-\int_{\p\Sigma}u^2\left(\frac{1}{\phi}\frac{\p\phi}{\p\nu}-\langle X,\nu\rangle\right)ds\\
&=\int_\Sigma\left(|\nabla u|^2+Qu^2\right)dv-\int_{\p\Sigma}u^2\left(\II^{\p M}(N,N)-\langle X,\nu\rangle\right)ds\\
&=\int_\Sigma\left(|\nabla u|^2+Qu^2\right)dv-\int_{\p\Sigma}u\frac{\p u}{\p\nu}ds\\
&=\int_\Sigma u\L_0u\hspace{.01cm}dv=\lambda_1(\L_0)\int_\Sigma u^2dv,
\end{align*}
where above we have used that $B\phi=0$ and $B_0u=0$. Thus $\lambda_1(\L_0)\ge0$.
\end{proof}

Our second auxiliary result is the following.

\begin{lemma}\label{lemma.9}
Let $(\Sigma,\gamma)$ be an $n$-dimensional, $n\ge2$, compact-with-boundary, connected, orientable, Riemannian manifold. Suppose there exists $u\in C^\infty(\Sigma)$, $u>0$, such that
\begin{align*}
&\L u\triangleq-\Delta u+\left(\frac{1}{2}R^\Sigma-P\right)u\ge0\,\,\,\,\mbox{on}\,\,\,\,\Sigma,\\
&\frac{\p u}{\p\nu}+H^{\p\Sigma}u\ge0\,\,\,\,\mbox{along}\,\,\,\,\p\Sigma,
\end{align*}
where $R^\Sigma$ is the scalar curvature of $(\Sigma,\gamma)$, $H^{\p\Sigma}$ is the mean curvature of $\p\Sigma$ in $(\Sigma,\gamma)$, and $P$ is a nonnegative function on $\Sigma$. Then either:
\begin{enumerate}
\item $\Sigma$ admits a metric of positive scalar curvature with minimal boundary or
\item $(\Sigma,\gamma)$ is Ricci flat with totally geodesic boundary, $P=0$, and $u$ is constant.
\end{enumerate}
\end{lemma}

\begin{proof}
Suppose $\Sigma$ does not admit a metric of positive scalar curvature with minimal boundary. Consider the metric $\hat\gamma=u^{2/(n-1)}\gamma$. The scalar curvature $R_{\hat\gamma}^\Sigma$ of $(\Sigma,\hat\gamma)$ and the mean curvature $H_{\hat\gamma}^{\p\Sigma}$ of $\p\Sigma$ in $(\Sigma,\hat\gamma)$ are given by
\begin{align}
R_{\hat\gamma}^\Sigma&=u^{-\frac{n+1}{n-1}}\left(-2\Delta u+R^\Sigma u+\frac{n}{n-1}\frac{|\nabla u|^2}{u}\right)=u^{-\frac{n+1}{n-1}}\left(2(\L u+Pu)+\frac{n}{n-1}\frac{|\nabla u|^2}{u}\right),\label{eq.4}\\
H_{\hat\gamma}^{\p\Sigma}&=u^{-\frac{n}{n-1}}\left(\frac{\p u}{\p\nu}+H^{\p\Sigma}u\right).\label{eq.5}
\end{align}
In particular, $R_{\hat\gamma}^\Sigma$ and $H_{\hat\gamma}^{\p\Sigma}$ are nonnegative. 

If $n=2$, by the Gauss-Bonnet theorem, $\chi(\Sigma)\ge0$. On the other hand, in dimension $n=2$, saying that $\Sigma$ does not admit a metric of positive scalar curvature (positive Gaussian curvature) with minimal boundary (geodesic boundary) is equivalent to saying that $\chi(\Sigma)\le0$ (see \cite[Theorem 1.2]{CruzVitorio}). Therefore $\chi(\Sigma)=0$ and, by the Gauss-Bonnet theorem, $R_{\hat\gamma}^\Sigma$ and $H_{\hat\gamma}^{\p\Sigma}$ vanish, since they are nonnegative. From \eqref{eq.4} and \eqref{eq.5}, we obtain that $u$ is constant and the functions $P$, $R^\Sigma$, and $H^{\p\Sigma}$ vanish. This finishes the proof in case $n=2$.

Assume that $n\ge3$.

{\bf Claim 1.} $R_{\hat\gamma}^\Sigma$ and $H_{\hat\gamma}^{\p\Sigma}$ vanish. If not, consider the eigenvalue problem
\begin{align}
-c_n\Delta_{\hat\gamma}\psi+R_{\hat\gamma}^\Sigma\psi&=\lambda\psi\,\,\,\,\mbox{on}\,\,\,\,\Sigma,\label{eq.6}\\
c_n\frac{\p\psi}{\p\nu_{\hat\gamma}}+2H_{\hat\gamma}^{\p\Sigma}\psi&=0\,\,\,\,\mbox{along}\,\,\,\,\p\Sigma,\nonumber
\end{align}
where $c_n=\frac{4(n-1)}{n-2}$, and let $\psi_1$ be an eigenfunction associated with the first eigenvalue $\lambda_1$ for this problem. We may assume that $\psi_1>0$. Multiplying \eqref{eq.6} by $\psi=\psi_1$ and integrating over $(\Sigma,\hat\gamma)$, we obtain
\begin{align*}
\lambda_1=\frac{\displaystyle\int_\Sigma\left(c_n|\nabla\psi_1|_{\hat\gamma}^2+R_{\hat\gamma}^\Sigma\psi_1^2\right)dv_{\hat\gamma}+2\int_{\p\Sigma}H_{\hat\gamma}^{\p\Sigma}\psi_1^2ds_{\hat\gamma}}{\displaystyle\int_\Sigma\psi_1^2dv_{\hat\gamma}}.
\end{align*}
Since $R_{\hat\gamma}^\Sigma$ and $H_{\hat\gamma}^{\p\Sigma}$ are nonnegative and, by hypothesis, at least one of them is not identically zero, we have $\lambda_1>0$. Therefore the scalar curvature of $(\Sigma,\gamma_1=\psi_1^{4/(n-2)}\hat\gamma)$ and the mean curvature of $\p\Sigma$ in $(\Sigma,\gamma_1)$ satisfy
\begin{align*}
R_{\gamma_1}^\Sigma&=\psi_1^{-\frac{n+2}{n-2}}(-c_n\Delta_{\hat\gamma}\psi_1+R_{\hat\gamma}^\Sigma\psi_1)=\psi_1^{-\frac{4}{n-2}}\lambda_1>0,\\
H_{\gamma_1}^{\p\Sigma}&=\frac{1}{2}\psi_1^{-\frac{n}{n-2}}\left(c_n\frac{\p\psi_1}{\p\nu}+2H_{\hat\gamma}^{\p\Sigma}\psi_1\right)=0,
\end{align*}
which is a contradiction, because we are assuming that $\Sigma$ does not admit a metric of positive scalar curvature with minimal boundary. Thus $R_{\hat\gamma}^\Sigma$ and $H_{\hat\gamma}^{\p\Sigma}$ vanish.

Using $R_{\hat\gamma}^\Sigma=0$ and $H_{\hat\gamma}^{\p\Sigma}=0$ into \eqref{eq.4} and \eqref{eq.5}, we obtain that $u$ is constant and the functions $P$, $R^\Sigma$, and $H^{\p\Sigma}$ vanish. Then we claim that $(\Sigma,\gamma)$ is Ricci flat with totally geodesic boundary. If not, it follows from \cite[Lemma 2.2]{CarlottoLi2021} that $\Sigma$ admits a metric of positive scalar curvature with mean convex boundary. In this case, acting exactly as in the proof of Claim 1, we get a contradiction. This finishes the proof.
\end{proof}

Now, we use Lemmas \ref{lemma.8} and \ref{lemma.9} to obtain an \sl{infinitesimal} rigidity result.

\begin{proposition}\label{prop.infinitesimal}
Let $(M,g,K)$ be an $(n+1)$-dimensional, $n\ge2$, initial data set with boundary and $\Sigma$ be a compact free boundary stable MOTS in $(M,g,K)$. Suppose $(M,g,K)$ satisfies the interior and the boundary DEC. If $\Sigma$ does not admit a metric of positive scalar curvature, then 
\begin{enumerate}
\item $\Sigma$ has vanishing outward null second fundamental form, i.e. $\chi^+=0$.
\item $\Sigma$ is Ricci flat and has totally geodesic boundary with respect to the induced metric.
\item The interior DEC saturates on $\Sigma$ and $J|_\Sigma=0$.
\item The boundary DEC saturates along $\p\Sigma$ and $(\iota_\varrho\pi)^\top|_{\p\Sigma}=0$.
\end{enumerate}
\end{proposition}

\begin{proof}
Let $u>0$ be an eigenfunction of $\L_0=-\Delta+Q$ on $\Sigma$, with Robin boundary condition $B_0u=0$, associated with the first eigenvalue $\lambda_1(\L_0)$, that is,
\begin{align*}
\L_0u&=-\Delta u+\left(\frac{1}{2}R^\Sigma-P\right)u=\lambda_1(\L_0)u\,\,\,\,\mbox{on}\,\,\,\,\Sigma,\\
\frac{\p u}{\p\nu}&=\left(\II^{\p M}(N,N)-\langle X,\nu\rangle\right)u\,\,\,\,\mbox{along}\,\,\,\,\p\Sigma,
\end{align*}
where 
\begin{align*}
P\triangleq\mu+J(N)+\frac{1}{2}|\chi^+|^2.
\end{align*}

In order to use Lemma \ref{lemma.9}, we need to check that $\L_0u$, $P$, and $\frac{\p u}{\p\nu}+H^{\p\Sigma}u$ are nonnegative. The first condition follows from Lemma \ref{lemma.8}, since $\Sigma$ is stable. The second one follows from the interior DEC. In fact,
\begin{align}\label{eq.7}
P\ge\mu+J(N)\ge\mu-|J|\ge0.
\end{align}
For the third, observe that 
\begin{align*}
H^{\p M}=H^{\p\Sigma}+\II(N,N)
\end{align*}
and 
\begin{align*}
(\iota_{\varrho}\pi)^\top(N)=K(\varrho,N)-(\tr K)\langle\varrho,N\rangle=K(\nu,N)=\langle X,\nu\rangle,
\end{align*}
since $\Sigma$ is free boundary. Therefore, using the boundary DEC, we obtain
\begin{align}
\frac{\p u}{\p\nu}+H^{\p\Sigma}u&=\left(H^{\p\Sigma}+\II^{\p M}(N,N)-\langle X,\nu\rangle\right)u\nonumber\\
&=\left(H^{\p M}-(\iota_{\varrho}\pi)^\top(N)\right)u\nonumber\\
&\ge\left(|(\iota_{\varrho}\pi)^\top|-(\iota_{\varrho}\pi)^\top(N)\right)u\nonumber\\
&\ge0.\label{eq.8}
\end{align}
Then, from Lemma \ref{lemma.9}, we have:
\begin{itemize}
\item $\Sigma$ is Ricci flat and has totally geodesic boundary with respect to the induced metric.
\item $P=0$. In particular, from \eqref{eq.7}, we obtain: $\chi^+=0$, $\mu=|J|$, and $J(N)=-|J|$. Also, observing that $|J|^2=|J|_\Sigma|^2+J(N)^2$, we get $J|_\Sigma=0$.
\item $u$ is constant. In this case, using that $\Sigma$ has totally geodesic boundary and, in particular, $H^{\p\Sigma}=0$, inequality \eqref{eq.8} must saturate, that is, $|(\iota_{\varrho}\pi)^\top|=(\iota_{\varrho}\pi)^\top(N)$. Finally, $|(\iota_{\varrho}\pi)^\top|^2=|(\iota_{\varrho}\pi)^\top|_{\p\Sigma}|^2+(\iota_{\varrho}\pi)^\top(N)^2$ gives $(\iota_{\varrho}\pi)^\top|_{\p\Sigma}=0$.
\end{itemize}
This finishes the proof.
\end{proof}

Consider the operator 
\begin{align*}
L^*\psi\triangleq-\Delta\psi-2\langle X,\nabla\psi\rangle+(Q-|X|^2-\div X)\psi,\,\,\,\,\psi\in C^\infty(\Sigma),
\end{align*}
called the \sl{formal adjoint} of $L$. Straightforward computations show that 
\begin{align*}
\int_\Sigma\psi L\phi\hspace{.01cm}dv+\int_{\p\Sigma}\psi B\phi\hspace{.01cm}ds=\int_\Sigma\phi L^*\psi\hspace{.01cm}dv+\int_{\p\Sigma}\phi B^*\psi\hspace{.01cm}ds
\end{align*}
for every $\phi,\psi\in C^\infty(\Sigma)$, where
\begin{align*}
B^*\psi\triangleq\frac{\p\psi}{\p\nu}-\left(\II^{\p M}(N,N)-2\langle X,\nu\rangle\right)\psi.
\end{align*}

\begin{remark}
It is important to note that, because $L$ is not necessarily symmetric and $\Sigma$ is a manifold with boundary, there is no \sl{standard} general existence result for the principal eigenvalue $\lambda_1(L)$ of $L$ (with Robin boundary condition $B\phi=0$) as in the closed case. Then, in order to guarantee the existence of $\lambda_1(L)$ with similar properties to those of the closed case, Alaee, Lesourd, and Yau \cite{AlaeeLesourdYau} assumed that $\II^{\p M}(N,N)\le0$ along $\p\Sigma$. But here, we succeed in circumventing this technical issue and do not make such an assumption.
\end{remark}

\begin{lemma}\label{lemma.formal.adjoint}
Under the assumptions of Proposition \ref{prop.infinitesimal}, zero is a simple eigenvalue of $L$ on $\Sigma$ with Robin boundary condition $B\phi=0$ whose associated eigenfunctions can be chosen to be positive. The same holds for the formal adjoint operator $L^*$ on $\Sigma$ with Robin boundary condition $B^*\phi^*=0$.
\end{lemma}

\begin{proof}
From Proposition \ref{prop.infinitesimal} (and its proof), we know that $Q=0$ and $\II^{\p M}(N,N)=\langle X,\nu\rangle$. Since $\Sigma$ is stable, we can take $\phi\in C^\infty(\Sigma)$, $\phi>0$, satisfying $L\phi\ge0$ on $\Sigma$ and $B\phi=0$ along~$\p\Sigma$. Then
\begin{align*}
0\le\frac{L\phi}{\phi}=-|X-\nabla\ln\phi|^2+\div(X-\nabla\ln\phi)
\end{align*}
provides, via the divergence theorem,
\begin{align*}
\int_\Sigma|X-\nabla\ln\phi|^2dv\le\int_{\p\Sigma}\left(\langle X,\nu\rangle-\frac{1}{\phi}\frac{\p\phi}{\p\nu}\right)ds=\int_{\p\Sigma}\left(\langle X,\nu\rangle-\II^{\p M}(N,N)\right)ds=0.
\end{align*}
Thus $X=\nabla\ln\phi$. In particular, $L\phi=0$. Therefore $\lambda=0$ is an eigenvalue of $L$ with boundary condition $B\phi=0$. To show that $\lambda=0$ is simple, let $\psi\in C^\infty(\Sigma)$ be such that $L\psi=0$ on $\Sigma$ with $B\psi=0$ along $\p\Sigma$ and define $u=\psi/\phi$. Simple computations show that  
\begin{align*}
\Delta u&=0\,\,\,\,\mbox{on}\,\,\,\,\Sigma,\\
\frac{\p u}{\p\nu}&=0\,\,\,\,\mbox{along}\,\,\,\,\p\Sigma.
\end{align*} 
Therefore $u=\psi/\phi$ is constant. Thus $\lambda=0$ is a simple eigenvalue whose associated eigenfunctions are given by $\psi=c\hspace{0.025cm}\phi$, $c\in\R$.

Now, observe that
\begin{align*}
L^*\phi^*&=-\Delta\phi^*-2\langle X,\nabla\phi^*\rangle-(|X|^2+\div X)\phi^*\\
&=-\Delta\phi^*-2\frac{\langle\nabla\phi,\nabla\phi^*\rangle}{\phi}-\left(\frac{|\nabla\phi|^2}{\phi^2}+\Delta\ln\phi\right)\phi^*\\
&=-\frac{\phi\Delta\phi^*+2\langle\nabla\phi,\nabla\phi^*\rangle+\phi^*\Delta\phi}{\phi}\\
&=-\frac{\Delta(\phi^*\phi)}{\phi}
\end{align*}
and 
\begin{align*}
B^*\phi^*&=\frac{\p\phi^*}{\p\nu}-\left(\II^{\p M}(N,N)-2\langle X,\nu\rangle\right)\phi^*\\
&=\frac{\p\phi^*}{\p\nu}+\II^{\p M}(N,N)\phi^*\\
&=\frac{\p\phi^*}{\p\nu}+\frac{\phi^*}{\phi}\frac{\p\phi}{\p\nu}\\
&=\frac{1}{\phi}\frac{\p}{\p\nu}(\phi^*\phi).
\end{align*}
Then $\phi^*\in C^\infty(\Sigma)$ satisfies $L^*\phi^*=0$ on $\Sigma$ with boundary condition $B^*\phi^*=0$ if, and only if, $\phi^*\phi$ is constant. Thus $\lambda^*=0$ is an eigenvalue of $L^*$ with Robin boundary condition $B^*\phi^*=0$ whose associated eigenfunctions are given by $\phi^*=c/\phi$, $c\in\R$. In particular, $\lambda^*=0$ is simple.
\end{proof}

The last auxiliary result of this section is a foliation lemma which we are going to use in the proof of Theorem \ref{thm.4}. It is similar to Lemma 2.3 in \cite{Galloway2018}.

\begin{lemma}\label{lemma.foliation}
Under the assumptions of Prop. \ref{prop.infinitesimal}, there exist a neighborhood $V\cong(-\epsilon,\epsilon)\times\Sigma$ of $\Sigma\cong\{0\}\times\Sigma$ in $M$ and a positive function $\varphi:V\to\R$ such that:
\begin{enumerate}
\item $g|_V$ has the orthogonal decomposition, 
\begin{align*}
g|_V=\varphi^2dt^2+\gamma_t,
\end{align*}
where $\gamma_t$ is the induced metric on $\Sigma_t\cong\{t\}\times\Sigma$.
\item Each $\Sigma_t$ is a free boundary hypersurface in $(M,g,K)$ with constant null mean curvature $\theta^+(t)$ with respect to the outward unit normal $N_t=\varphi^{-1}\frac{\p}{\p t}$, where $N_0=N$.
\item $\frac{\p\varphi}{\p\nu_t}=\II^{\p M}(N_t,N_t)\varphi$ along $\p\Sigma_t$, where $\nu_t$ is the outward unit normal of $\p\Sigma_t$ in $(\Sigma_t,\gamma_t)$.\label{condition.3}
\end{enumerate}
\end{lemma}

To prove Lemma \ref{lemma.foliation}, we use ideas presented in \cite{AlaeeLesourdYau}, \cite{Ambrozio}, and \cite{Galloway2008}.

\begin{proof}
Let $Z$ be a smooth vector field on $M$ such that $Z_x=N_x$ for $x\in\Sigma$ and $Z_p\in T_p\p M$ for $p\in\p M$. Denote by $\Phi=\Phi(p,t)$ the flow of $Z$. Since $\Sigma$ is compact, there exists $\delta>0$ such that $\Phi(x,t)$ is well-defined for $(x,t)\in\Sigma\times(-\delta,\delta)$.

Fix $0<\alpha<1$ and define $B_\delta(0)=\{u\in C^{2,\alpha}(\Sigma);\|u\|_{2,\alpha}<\delta\}$. Given $u\in B_\delta(0)$, consider $\Sigma_u=\{\Phi(x,u(x));x\in\Sigma\}$ and let $\theta_u^+$ denote the null mean curvature of $\Sigma_u$ with respect to the (suitably chosen) outward unit normal $N_u$ of $\Sigma_u$. (Taking a smaller $\delta>0$ if necessary, we may assume that $\Sigma_u$ is a properly embedded hypersurface in $M$ for each $u\in B_\delta(0)$.) 

Define $\Theta:B_\delta(0)\times\R\to C^{0,\alpha}(\Sigma)\times C^{1,\alpha}(\p\Sigma)\times\R$ by 
\begin{align*}
\Theta(u,k)=\left(\theta_u^+-k,-\langle N_u,\varrho_u\rangle,\int_\Sigma udv\right),
\end{align*}
where $\varrho_u$ is the restriction of $\varrho$ to $\p\Sigma_u$. It follows from \eqref{eq.first.variation.theta} and \cite[Proposition 17]{Ambrozio} that the linearization of $\Theta$ at $(0,0)$ is given by
\begin{align*}
D\Theta_{(0,0)}(u,k)=\frac{d}{ds}\big|_{s=0}\Theta(su,sk)=\left(Lu-k,Bu,\int_\Sigma udv\right),\,\,\,\,(u,k)\in C^{2,\alpha}(\Sigma)\times\R.
\end{align*} 

{\bf Claim 1.} $D\Theta_{(0,0)}$ is an isomorphism. First, observe that if $(u,k)\in\mbox{ker}(D\Theta_{(0,0)})$, then
\begin{align}
Lu&=k\,\,\,\,\mbox{on}\,\,\,\,\Sigma,\label{eq.9}\\
Bu&=0\,\,\,\,\mbox{along}\,\,\,\,\p\Sigma,\nonumber\\
\int_\Sigma u&dv=0.\nonumber
\end{align}
Now, let $\phi^*>0$ be an eigenfunction of $L^*$ on $\Sigma$, with boundary condition $B^*\phi^*=0$, associated with the eigenvalue $\lambda^*=0$, which exists by Lemma \ref{lemma.formal.adjoint}. Then, multiplying \eqref{eq.9} by $\phi^*$ and integrating over $\Sigma$, we obtain
\begin{align*}
k\int_\Sigma\phi^*dv=\int_\Sigma\phi^*Ludv=\int_\Sigma uL^*\phi^*dv+\int_{\p\Sigma}uB^*\phi^*ds-\int_{\p\Sigma}\phi^*Buds=0.
\end{align*}
Thus $k=0$. In this case, $Lu=0$ on $\Sigma$, $Bu=0$ along $\p\Sigma$, and $\int_\Sigma udv=0$. From Lemma~\ref{lemma.formal.adjoint}, $u=c\hspace{0.025cm}\phi$, $c\in\R$, where $\phi>0$ is an eigenfunction of $L$ on $\Sigma$, with Robin boundary condition $B\phi=0$, associated with the eigenvalue $\lambda=0$. Since $\int_\Sigma udv=0$, we have $u=0$. This shows that $D\Theta_{(0,0)}$ is injective. 

To see that $D\Theta_{(0,0)}$ is onto, observe that, by the Fredholm alternative, the problem
\begin{align*}
\left\{
\begin{array}{rcl}
Lu&\!\!\!=\!\!\!&f\,\,\,\,\mbox{on}\,\,\,\,\Sigma\\
Bu&\!\!\!=\!\!\!&h\,\,\,\,\mbox{along}\,\,\,\,\p\Sigma
\end{array}
\right.
\end{align*}
has a solution if, and only if, 
\begin{align*}
\int_\Sigma\phi^*fdv+\int_{\p\Sigma}\phi^*hds=0.
\end{align*}
Therefore, given $(f,h,c)\in C^{0,\alpha}(\Sigma)\times C^{1,\alpha}(\p\Sigma)\times\R$, we can take
\begin{align*}
k_0=-\frac{\displaystyle\int_\Sigma\phi^*fdv+\int_{\p\Sigma}\phi^*hds}{\displaystyle\int_\Sigma\phi^*dv},
\end{align*}
$u_0\in C^{2,\alpha}(\Sigma)$ such that $Lu_0=k_0+f$ on $\Sigma$ and $Bu_0=h$ along $\p\Sigma$, which exists by the Fredholm alternative, and 
\begin{align*}
t_0=\frac{c-\displaystyle\int_\Sigma u_0dv}{\displaystyle\int_\Sigma\phi dv}.
\end{align*}
Straightforward computations show that 
\begin{align*}
D\Theta_{(0,0)}(u_0+t_0\phi,k_0)=(f,h,c).
\end{align*}
This finishes the proof of Claim 1.

It follows from the inverse function theorem that, for $s\in\R$ sufficiently small, say $|s|<\varepsilon$, there exist $u(s)\in B_\delta(0)$ and $k(s)\in\R$ such that $\Theta(u(s),k(s))=(0,0,s)$ with $u(0)=0$ and $k(0)=0$. By the chain rule, 
\begin{align*}
\left(Lu'(0)-k'(0),Bu'(0),\int_\Sigma u'(0)dv\right)=D\Theta_{(0,0)}(u'(0),k'(0))=\frac{d}{ds}\big|_{s=0}\Theta(u(s),k(s))=(0,0,1).
\end{align*}
Therefore, using the same arguments as before, we have $k'(0)=0$ and $u'(0)=c\hspace{0.025cm}\phi$, $c\in\R$. Observe that $c>0$ since $\int_\Sigma u'(0)dv=1$. Then, taking a smaller $\varepsilon>0$ if necessary, we obtain that $(\Sigma_{u(s)})_{|s|<\varepsilon}$ forms a foliation of a neighborhood of $\Sigma$ in $M$ by free boundary hypersurfaces with constant null mean curvature.

Finally, we can introduce coordinates $(t,x^i)$ in a neighborhood $V$ of $\Sigma$ such that, with respect to these coordinates, $V\cong(-\epsilon,\epsilon)\times\Sigma$, each slice $\Sigma_t\cong\{t\}\times\Sigma$ is a free boundary hypersurface in $M$ with constant null mean curvature $\theta^+(t)$ with respect to the outward unit normal $N_t$, and $\Sigma=\Sigma_0$. Furthermore, these coordinates can be chosen in a such way that $\frac{\p}{\p t}=\varphi N_t$ on $\Sigma_t$ for some positive function $\varphi:V\to\R$. Item \eqref{condition.3} follows from the free boundary condition (see \cite[Proposition 17]{Ambrozio}).
\end{proof}

\section{Proof of Theorem \ref{thm.4}}\label{section.4}

Let $\varphi$, $\Sigma_t$, and $\theta^+(t)$, $t\in(-\epsilon,\epsilon)$, be as in Lemma \ref{lemma.foliation}. From \eqref{eq.first.variation.theta},
\begin{align}\label{eq.10}
\frac{d\theta^+}{dt}=-\Delta\varphi+2\langle X,\nabla\varphi\rangle+\left(Q-|X|^2+\div X-\frac{1}{2}(\theta^+)^2+\theta^+\tau\right)\varphi,
\end{align}
where $\Delta=\Delta_t$, $X=X_t$, $\nabla=\nabla_t$, $Q=Q_t$, and $\div=\div_t$ are the respective entities associated with $\Sigma_t$ for each $t\in(-\epsilon,\epsilon)$. Taking a smaller $\epsilon>0$ if necessary, we may assume that $\tau\varphi\le c$ on $U\cong[0,\epsilon)\times\Sigma$ for some constant $c>0$. Observe that $\theta^+(t)\ge0$ for each $t\in[0,\epsilon)$, since $\Sigma$ is weakly outermost. Then, from \eqref{eq.10}, we have
\begin{align*}
L_t\varphi\triangleq-\Delta\varphi+2\langle X,\nabla\varphi\rangle+(Q-|X|^2+\div X)\varphi\ge\frac{d\theta^+}{dt}-c\hspace{0.0375cm}\theta^+=e^{c\hspace{0.025cm}t}\frac{d}{dt}(e^{-c\hspace{0.025cm}t}\theta^+)\,\,\,\,\mbox{on}\,\,\,\,\Sigma_t
\end{align*}
for each $t\in[0,\epsilon)$.

{\bf Claim 1.} $\frac{d}{dt}(e^{-c\hspace{0.025cm}t}\theta^+)\le0$ for $t\in[0,\epsilon)$. Suppose, by contradiction, that $\frac{d}{dt}(e^{-c\hspace{0.025cm}t}\theta^+)>0$ for some $t\in[0,\epsilon)$. In particular, $L_t\varphi>0$ on $\Sigma_t$. Then, acting as in the proof of Lemma \ref{lemma.8}, we obtain
\begin{align}\label{eq.11}
\frac{L_t\varphi}{\varphi}=-|Y|^2+\div Y+Q,
\end{align}
where $Y=X-\nabla\ln\varphi$, and 
\begin{align*}
u^2\frac{L_t\varphi}{\varphi}\le\div(u^2Y)+|\nabla u|^2+Qu^2
\end{align*}
for any smooth function $u$ on $\Sigma_t$. 

Let $u=u_t>0$ be an eigenfunction of $\L_t\triangleq-\Delta+Q$ on $\Sigma_t$, with boundary condition $\frac{\p u_t}{\p\nu_t}=\left(\II^{\p M}(N_t,N_t)-\langle X,\nu_t\rangle\right)u_t$, associated with the first eigenvalue $\lambda_1(\L_t)$. Then, acting as in the proof of Lemma \ref{lemma.8} and using that $L_t\varphi>0$, we have
\begin{align*}
0<\int_{\Sigma_t}u_t^2\frac{L_t\varphi}{\varphi}dv_t\le\int_{\Sigma_t}\left(|\nabla u_t|^2+Qu_t^2+\div(u_t^2Y)\right)dv_t=\lambda_1(\L_t)\int_{\Sigma_t}u_t^2dv_t.
\end{align*} 
In particular, $\lambda_1(\L_t)>0$. Therefore 
\begin{align*}
\L_tu_t=-\Delta u_t+\left(\frac{1}{2}R^{\Sigma_t}-P_t\right)u_t=\lambda_1(\L_t)u_t>0\,\,\,\,\mbox{on}\,\,\,\,\Sigma_t,
\end{align*}
where 
\begin{align*}
P_t\triangleq\mu+J(N_t)+\frac{1}{2}|\chi_t^+|^2.
\end{align*}
Using the interior and the boundary DEC, and acting as in the proof of Proposition \ref{prop.infinitesimal}, we obtain
\begin{align*}
P_t\ge0\,\,\,\,\mbox{and}\,\,\,\,\frac{\p u_t}{\p\nu_t}+H^{\p\Sigma_t}u_t\ge0.
\end{align*}
Then, from Lemma \ref{lemma.9}, $(\Sigma_t,\gamma_t)$ is Ricci flat, $P_t=0$, and $u_t$ is constant. In particular, $\L_tu_t=0$, which is a contradiction because $\L_tu_t=\lambda_1(\L_t)u_t>0$. Here we have used that $\Sigma_t\cong\{t\}\times\Sigma$ does not admit a metric of positive scalar curvature with minimal boundary. This finishes the proof of Claim 1.

Claim 1 gives that $e^{-c\,t}\theta^+(t)\le\theta^+(0)=0$ for $t\in[0,\epsilon)$. Then $\theta^+(t)=0$ for every $t\in[0,\epsilon)$, since $\Sigma$ is weakly outermost. Using this information into \eqref{eq.10}, we obtain $L_t\varphi=0$ for each $t\in[0,\epsilon)$. In particular, each such $\Sigma_t$ is a stable MOTS. Theorem \ref{thm.4} then follows from Proposition \ref{prop.infinitesimal} applied to $\Sigma_t$ for each $t\in[0,\epsilon)$.

\section{A Splitting Result}\label{section.5}

The class of initial data sets $(M,g,K)$ covered by Theorem \ref{thm.4} is expected to be fairly wide (see Example 4.2 in \cite{EichmairGallowayMendes}). Therefore, if we want to prove a stronger rigidity result, it is natural to assume some extra conditions. In what follows, we show (among other things) that, under natural \sl{convexity} and \sl{volume-minimizing} conditions, $(M,g)$ splits in an outer neighborhood of $\Sigma$.

We say that $K$ is \sl{$n$-convex} with respect to $g$ if, at every point of $M$, the sum of the smallest $n$ eigenvalues of $K$ with respect to $g$ is nonnegative (see \cite{EichmairGallowayMendes,Mendes}). In particular, if $K$ is $n$-convex, then $\tr_S K\ge0$ for every hypersurface $S\subset M$. We say that $\Sigma$ is \sl{outer volume-minimizing} if $\mbox{vol}(\Sigma)\le\mbox{vol}(S)$ for every hypersurface $S$ outside of, and homologous to, $\Sigma$.

\begin{thm}\label{thm.14}
Under the assumptions of Theorem \ref{thm.4}, if $K$ is $n$-convex and $\Sigma$ is outer volume-minimizing, then there exists an outer neighborhood $U$ of $\Sigma$ in $M$ such that the following conditions hold: 
\begin{enumerate}
\item $(\Sigma,\gamma_0)$ is Ricci flat with totally geodesic boundary, where $\gamma_0$ is the metric on $\Sigma$ induced from $(M,g)$.
\item $(U,g|_U)$ is isometric to $([0,\epsilon)\times\Sigma,dt^2+\gamma_0)$.
\item $K=a\hspace{0.025cm}dt^2$ on $U$, where $a$ depends only on $t\in[0,\epsilon)$.
\item $\mu=0$ and $J=0$ on $U$.
\item $\II^{\p M}=0$ and $(\iota_\varrho\pi)^\top=0$ on $U\cap\p M$.
\end{enumerate}
\end{thm}

\begin{proof}
Let $\varphi$ and $\Sigma_t$, $t\in[0,\epsilon)$, be as in the proof of Theorem \ref{thm.4}. Remember that $L_t\varphi=0$, $\chi_t^+=0$, $(\Sigma_t,\gamma_t)$ is Ricci flat with totally geodesic boundary, and $\mu=|J|=-J(N_t)$ on $\Sigma_t$ (see the proof of Proposition \ref{prop.infinitesimal}). In particular, 
\begin{align*}
Q=\frac{1}{2}R^{\Sigma_t}-(\mu+J(N_t))-\frac{1}{2}|\chi_t^+|^2=0.
\end{align*}
From equation \eqref{eq.11}, we have
\begin{align*}
\int_{\Sigma_t}|X-\nabla\ln\varphi|^2dv_t&=\int_{\Sigma_t}\div(X-\nabla\ln\varphi)dv_t\\
&=\int_{\p\Sigma_t}\left(\langle X,\nu_t\rangle-\frac{1}{\varphi}\frac{\p\varphi}{\p\nu_t}\right)ds_t\\
&=\int_{\p\Sigma_t}\left(\langle X,\nu_t\rangle-\II^{\p M}(N_t,N_t)\right)ds_t,
\end{align*}
where above we have used item \eqref{condition.3} of Lemma \ref{lemma.foliation}. On the other hand, it follows from the proof of Proposition \ref{prop.infinitesimal} that
\begin{align*}
\II^{\p M}(N_t,N_t)=H^{\p M}=|(\iota_\varrho\pi)^\top|=(\iota_\varrho\pi)^\top(N_t)=\langle X,\nu_t\rangle\,\,\,\,\mbox{along}\,\,\,\,\p\Sigma_t.
\end{align*}
Therefore
\begin{align*}
\int_{\Sigma_t}|X-\nabla\ln\varphi|^2dv_t=0,
\end{align*}
i.e.
\begin{align*}
X=\nabla\ln\varphi\,\,\,\,\mbox{on}\,\,\,\,\Sigma_t.
\end{align*}

Now, observing that $\tr_{\Sigma_t}K\ge0$, since $K$ is $n$-convex, we obtain
\begin{align*}
H(t)\le\tr_{\Sigma_t}K+H(t)=\theta^+(t)=0,
\end{align*}
where $H(t)=\div_{\Sigma_t}N_t$ is the mean curvature of $\Sigma_t$ in $(M,g)$. Then the first variation formula of volume (see \cite[Proposition 14]{Ambrozio}) gives
\begin{align*}
\frac{d}{dt}\mbox{vol}(\Sigma_t)=\int_{\Sigma_t}\varphi H(t)dv_t+\int_{\p\Sigma_t}\varphi\langle N_t,\nu_t\rangle ds_t=\int_{\Sigma_t}\varphi H(t)dv_t\le0.
\end{align*}
Therefore $t\longmapsto\mbox{vol}(\Sigma_t)$, $t\in[0,\epsilon)$, is a nonincreasing function. In particular, $\mbox{vol}(\Sigma_t)\le\mbox{vol}(\Sigma_0)$ for each $t\in[0,\epsilon)$. Thus, because we are assuming that $\Sigma=\Sigma_0$ is outer volume-minimizing, it follows that $\mbox{vol}(\Sigma_t)=\mbox{vol}(\Sigma_0)$ and, in particular, $H(t)=\tr_{\Sigma_t}K=0$, for each $t\in[0,\epsilon)$.

Remember that
\begin{align*}
g|_U=\varphi^2dt^2+\gamma_t,
\end{align*}
where $U\cong[0,\epsilon)\times\Sigma$. We claim that each leaf $\Sigma_t\cong\{t\}\times\Sigma$ is totally geodesic in $(M,g)$ and $\varphi$ is constant on $\Sigma_t$, i.e. $\varphi=\varphi(t)$ depends only on $t\in[0,\epsilon)$. Provided that is the case, it is not difficult to see that $\gamma_t$ does not depend on $t$ and, after the simple change of variable $ds=\varphi(t)dt$, the metric $g$ has the product structure $ds^2+\gamma_0$ on $U$, where $(\Sigma,\gamma_0)$ is Ricci flat with totally geodesic boundary, which guarantees items (1) and (2). In order to prove the claim, we consider the first variation of the \sl{inward} null mean curvature $\theta^-(t)$ of $\Sigma_t$ in $(M,g,K)$.

The first variation of $\theta^-(t)=\tr_{\Sigma_t}K-H(t)=0$ is given by
\begin{align}\label{eq.12}
-\Delta\varphi^-+2\langle X^-,\nabla\varphi^-\rangle+(Q^--|X^-|^2+\div X^-)\varphi^-=\frac{d\theta^-}{dt}=0,
\end{align} 
where $\varphi^-=-\varphi$,
\begin{align}\label{eq.13}
Q^-=\frac{1}{2}R^{\Sigma_t}-(\mu+J(-N_t))-\frac{1}{2}|\chi_t^-|^2=-2|J|-\frac{1}{2}|\chi_t^-|^2,
\end{align}
and $X^-$ is the tangent vector field of $\Sigma_t$ that is dual to the $1$-form $K(-N_t,\cdot)|_{\Sigma_t}=-K(N_t,\cdot)|_{\Sigma_t}$, i.e.
\begin{align}\label{eq.14}
X^-=-X=-\nabla\ln\varphi\,\,\,\,\mbox{on}\,\,\,\,\Sigma_t.
\end{align}
Substituting \eqref{eq.13} and \eqref{eq.14} into \eqref{eq.12}, with $\varphi^-=-\varphi$, we get
\begin{align*}
\Delta\varphi+\frac{|\nabla\varphi|^2}{\varphi}+\left(|J|+\frac{1}{4}|\chi_t^-|^2\right)\varphi=0.
\end{align*}
Therefore, integrating over $\Sigma_t$ and using that
\begin{align*}
\frac{\p\varphi}{\p\nu_t}=\II^{\p M}(N_t,N_t)\varphi=|(\iota_\varrho\pi)^\top|\varphi\ge0,
\end{align*}
we obtain
\begin{align*}
0\le\int_{\p\Sigma_t}\frac{\p\varphi}{\p\nu_t}ds_t=\int_{\Sigma_t}\Delta\varphi\hspace{0.025cm}dv_t=-\int_{\Sigma_t}\left(\frac{|\nabla\varphi|^2}{\varphi}+\left(|J|+\frac{1}{4}|\chi_t^-|^2\right)\varphi\right)dv_t\le0,
\end{align*}
that is,
\begin{align}\label{eq.15}
|\nabla\varphi|=|J|=|\chi_t^-|=0\,\,\,\,\mbox{on}\,\,\,\,\Sigma_t.
\end{align}
In particular, $\varphi$ is constant on $\Sigma_t$. On the other hand, $\chi_t^+=\chi_t^-=0$ gives that $K|_{\Sigma_t}=0$ and $\Sigma_t$ is totally geodesic in $(M,g)$.

Observe that $K(N_t,\cdot)|_{\Sigma_t}=0$, since $X=\nabla\ln\varphi=0$. Therefore, because $K|_{\Sigma_t}=0$ for each $t\in[0,\epsilon)$, we have $K=a\hspace{0.025cm}dt^2$ on $U\cong[0,\epsilon)\times\Sigma$. Then we can use that $d\tr K=\div K$ (since $J=0$) to see that $a$ is constant on each leaf $\Sigma_t$. This proves item (3).

Item (4) follows directly from equation \eqref{eq.15} and the fact that $\mu=|J|$ on each leaf $\Sigma_t$.

Finally, since $|(\iota_\varrho\pi)^\top|=\II^{\p M}(N_t,N_t)=\langle X,\nu_t\rangle=0$ along $\p\Sigma_t$, $\p\Sigma_t$ is totally geodesic in $(\Sigma_t,\gamma_t)$, and $\Sigma_t$ is free boundary in $(M,g)$, we obtain that $(\iota_\varrho\pi)^\top=0$ and $\II^{\p M}=0$ along $\p\Sigma_t$ for each $t\in[0,\epsilon)$. This finishes the proof of Theorem \ref{thm.14}
\end{proof}

A similar splitting result to Theorem \ref{thm.14} was established by M. Eichmair, G. J. Galloway, and the author \cite{EichmairGallowayMendes} for closed MOTS $\Sigma$ in $n=2,\ldots,6$ dimensions. It is important to mention that they \sl{did not} suppose $\Sigma$ is weakly outermost. Instead, under a suitable barrier condition, they \sl{proved} that $\Sigma$ is weakly outermost and, in particular, stable. (Example 5.3 in \cite{EichmairGallowayMendes} shows that the convexity and the volume-minimizing conditions are in fact necessary for the splitting result.) 

Alaee, Lesourd, and Yau \cite{AlaeeLesourdYau} established an analogous result to Theorem~\ref{thm.14} in $n=2$ dimensions, provided $\mu-|J|$ is bounded from below by a positive constant. As far as we know, they were the pioneers in considering this kind of problem for free boundary MOTS.

\bibliographystyle{amsplain}
\bibliography{bibliography.bib}

\end{document}